\documentclass[a4paper,11pt]{amsart}
\usepackage{fancyhdr}
\usepackage{amsmath}
\usepackage{dsfont}
\usepackage{hyperref}
\usepackage[mathscr]{eucal}
\usepackage[cp1251]{inputenc}
\usepackage[english]{babel}
\usepackage{enumerate,float,indentfirst}
\usepackage{graphicx}
\usepackage{xcolor}
\usepackage{latexsym,a4,mathrsfs,amsthm,amsmath,amssymb,url}
\usepackage{amsfonts}
\usepackage{amssymb}
\usepackage{caption}

\usepackage{stackengine}

\numberwithin{equation}{section}
\newtheorem{lemma}{Lemma}

\newtheorem{corollary}{Corollary}
\newtheorem{proposition}{Proposition}

\begin{document}

\vspace{0in}

\title[Submatrices with the best-bounded inverses]{\bf Submatrices with the best-bounded inverses: the equality criterion for $\mathbb{R}^{n \times 2}$ }

\author[Yu. Nesterenko]{Yuri Nesterenko}
%\address{ Individual contributor }
\email{yuri.r.nesterenko@gmail.com}

\begin{abstract}
The long-standing hypothesis formulated by Goreinov, Tyrtyshnikov and Zamarashkin \cite{GTZ1997} has recently been solved positively by Sengupta and Pautov \cite{SP2026} in the case of two-column matrices. In this paper, we complement their elegant proof with the equality criterion.
\end{abstract}

\maketitle

\thispagestyle{empty}
\vspace{-5truemm}
\section{Introduction}

The long-standing hypothesis formulated by Goreinov, Tyrtyshnikov and Zamarashkin \cite{GTZ1997} has recently been solved positively in the case of two-column matrices by  Sengupta and Pautov \cite{SP2026}.
In this paper, we reproduce their elegant proof complementing it with the equality criterion.

Throughout, we will use the notion of square of a vector from $\mathbb R^{2}$ defined in the same way as for complex numbers:
\begin{equation*}
(\,\xi, \; \eta \,) \mapsto (\, \xi^2-\eta^2, \; 2\xi\eta \,).
\end{equation*}

\begin{proposition}\label{pr}
For every $n \geq 3$ and an arbitrary real $n \times 2$ matrix with orthonormal columns, there exists a $2 \times 2$ submatrix such that the spectral norm of its inverse does not exceed $\sqrt{n}$. The bound is tight and the corresponding equality for a fixed $n$ occurs if and only if the squares of the rows of the matrix are clustered into three non-empty groups of identical vectors of magnitudes
\begin{equation*}
\frac{1}{2n} + \frac{1}{2p}, \quad \frac{1}{2n} + \frac{1}{2q} \quad \text{ and } \quad \frac{1}{2n} + \frac{1}{2r},
\end{equation*}
where $p, \, q, \, r > 0, \; p + q + r = n$ are the corresponding cluster sizes.
\end{proposition}

This yields the following equivalent statement about planar polygons conjectured in \cite{Nesterenko2024}.
\begin{corollary}\label{co}
For arbitrary vectors $w_1, \ldots, w_n, \, n \geq 3$ forming a planar polygon of perimeter 2
\begin{equation*}
\begin{split}
1) \, & \max_{i \neq j} (|w_i| + |w_j| - |w_i + w_j|) \geq \frac{2}{n}, \\
2) \, &\text{the bound is tight,} \\
3) \, & \text{equality holds} \iff \text{the multiset } \{ w_1, \ldots, w_n \} \text{ consists of three distinct} \\
& \text{vectors } x, y, z \text{ of magnitudes } \frac{1}{2n} + \frac{1}{2p}, \, \frac{1}{2n} + \frac{1}{2q}, \frac{1}{2n} + \frac{1}{2r},
\text{ respectively,} \\
& \text{corresponding to their numbers of occurrences } p, \, q, \, r > 0, \; p + q + r = n. 
\end{split}
\end{equation*}
\end{corollary}

The correspondence between these two formulations can be established by computing the squares of the rows of the matrix. In particular,
\begin{equation}\label{rel}
\sigma_2^2\left( \begin{matrix}
\text{-- } a_i \text{ --} \\
\text{-- } a_j \text{ --}
\end{matrix} \right) = \frac{|w_i| + |w_j| - | w_i + w_j |}{2},
\end{equation}
where $w_i = a_i^2$, $w_j = a_j^2$ and $\sigma_2$ is the smallest singular value of a $2 \times 2$ matrix (see \cite{Nesterenko2024}).

\section{Proof}

We will need the following lemma.
\begin{lemma}\label{le}
Let $v_1, v_2, v_3 \in \mathbb{R}^2$ be vectors with pairwise non-collinear squares. Then at least two of the three $2 \times 2$ matrices composed of these vectors as rows have pairwise different first (and second) right-singular vectors.
\end{lemma}

\textbf{Proof.} First, let us show that the mentioned right-singular vectors are uniquely defined up to signs. Indeed, since the square vectors are pairwise non-collinear, the vectors themselves are non-orthogonal and hence the singular values of the matrices are distinct.

By the explicit formula from \cite{B1996}, the singular value decomposition of matrix whose rows are $v_i$ and $v_j$ can be written as
\begin{equation*}
\left( \begin{matrix}
v_{i1} & v_{i2} \\
v_{j1} & v_{j2}
\end{matrix} \right) =
\left( \begin{matrix}
\cos \alpha & \sin \alpha \\
-\sin \alpha & \cos \alpha
\end{matrix} \right)
\left( \begin{matrix}
\sigma_1 & 0 \\
0 & \sigma_2
\end{matrix} \right) 
\left( \begin{matrix}
\cos \beta & \sin \beta \\
-\sin \beta & \cos \beta
\end{matrix} \right),
\end{equation*}
where
\begin{equation*}
2\beta = \mathrm{atan2}\,(\, 2 v_{i1} v_{i2} + 2 v_{j1} v_{j2}, \; v_{i1}^2 - v_{i2}^2 + v_{j1}^2 - v_{j2}^2 \,).
\end{equation*}
Thus, the first right-singular vector direction is given by the square root of the sum of squares of $v_i$ and $v_j$.

Since $v_1^2$, $v_2^2$ and $v_3^2$ are pairwise non-collinear, at least two of their sums $v_1^2 + v_2^2$, $v_1^2 + v_3^2$ and $v_2^2 + v_3^2$ are non-collinear too, which finalizes the proof of the lemma.

%\qedsymbol

\quad

We will now reproduce the proof from \cite{SP2026}, supplementing it with a consideration of the equality criterion, thereby proving Proposition \ref{pr}.

The base case $n = 3$ can be easily proven using formula \eqref{rel}. Let us assume that the statement is true for $n-1$ rows and fix the matrix
\begin{equation*}
A = \left( \begin{matrix}
a_{11} & a_{12} \\
a_{21} & a_{22} \\
\dots & \dots \\
a_{n1} & a_{n2}
\end{matrix} \right) \in \mathbb{R}^{n \times 2}, \quad A^T A = I.
\end{equation*}

Similarly to \cite{SP2026}, we consider the following two possibilities.

\textbf{Case A.} At least one row of the matrix $A$ has a small norm. Specifically,
\begin{equation*}
\exists i: a_{i1}^2 + a_{i2}^2 \leq \frac{1}{n}.
\end{equation*}

Without loss of generality, assume $i = 1$. Let us show that in this case the corresponding inequality is not only true but also strict.

Introduce rotation matrix $P \in \mathbb{R}^{2 \times 2}$ such that
\begin{equation*}
B := AP = \left( \begin{matrix}
b & 0 \\
b_{21} & b_{22} \\
\dots & \dots \\
b_{n1} & b_{n2}
\end{matrix} \right), \text{ where } b^2 = b_{11}^2 + b_{12}^2 \leq \frac{1}{n}.
\end{equation*}
Proving the statement for matrix $B$ is equivalent to proving it for $A$. The case $b = 0$ follows immediately from the inductive hypothesis, so let us concentrate on the case $b \neq 0$.

Put $t = 1/\sqrt{1 - b^2} > 1$ and define the matrix
\begin{equation*}
\tilde{B} = \left( \begin{matrix}
t b_{21} & b_{22} \\
\dots & \dots \\
t b_{n1} & b_{n2}
\end{matrix} \right) \in \mathbb R^{(n-1) \times 2}.
\end{equation*}

The choice of $t$ yields
\begin{equation*}
\tilde{B}^T \tilde{B} = I.
\end{equation*}

Thus by the inductive hypothesis, there exists a $2 \times 2$ submatrix of $\tilde{B}$ with the spectral norm of its inverse less than or equal to $\sqrt{n-1}$. Assume that
\begin{equation*}
\tilde{B}_{ij} = \left( \begin{matrix}
t b_{i1} & b_{i2} \\
t b_{j1} & b_{j2}
\end{matrix} \right) =
\left( \begin{matrix}
b_{i1} & b_{i2} \\
b_{j1} & b_{j2}
\end{matrix} \right)
\left( \begin{matrix}
t & 0 \\
0 & 1
\end{matrix} \right)
\end{equation*}
is such a submatrix, with the smallest spectral norm of inverse among all the $2 \times 2$ submatrices.

Since $||\tilde{B}_{ij}^{-1}||_2 = \sigma_2^{-1}(\tilde{B}_{ij})$, applying the min--max characterization of singular values (see, e.g., \cite{HJ2012}), we get
\begin{equation*}
\sigma_2\left( \begin{matrix}
b_{i1} & b_{i2} \\
b_{j1} & b_{j2}
\end{matrix} \right) \geq \sigma_2(\tilde{B}_{ij}) \, \sigma_2\left( \begin{matrix}
1/t & 0 \\
0 & 1
\end{matrix} \right) \geq \sqrt{\frac{1}{n-1}} \times \frac{1}{t}.
\end{equation*}

This implies
\begin{equation}\label{ineq}
\sigma_2^2\left( \begin{matrix}
b_{i1} & b_{i2} \\
b_{j1} & b_{j2}
\end{matrix} \right) \geq \frac{1}{n-1} \times \frac{1}{t^2} = \frac{1-b^2}{n-1} \geq \frac{1 - 1/n}{n-1} = \frac{1}{n}.
\end{equation}

The equality in \eqref{ineq} is achieved if and only if the following three conditions are met.

(i) $b^2 = \frac{1}{n}$,

(ii) $\sigma_2(\tilde{B}_{ij}) = \frac{1}{\sqrt{n-1}}$,

(iii) vector $(\pm 1, 0)^T$ (being the smallest left-singular vector of matrix $\operatorname{diag}(1/t, 1)$) can be taken as a smallest right-singular vector of $\tilde{B}_{ij}$.

By the inductive hypothesis, the second condition implies that matrix $\tilde{B}$ satisfies the equality criterion from the Proposition \ref{pr}, and hence contains three rows satisfying Lemma \ref{le}. These rows form three $2 \times 2$ submatrices of $\tilde{B}$ with the smallest singular value equal to $1/\sqrt{n-1}$. By Lemma \ref{le},
there is at least one of these submatrices, for which $(\pm 1, 0)^T$ is not the smallest right-singular vector. Assume $\tilde{B}_{kl}$ is such a submatrix. Taking  $\tilde{B}_{kl}$ instead of $\tilde{B}_{ij}$ in the considerations above, we break point (iii) and hence establish the validity of the strict version of \eqref{ineq}
\begin{equation*}
\sigma_2\left(\begin{matrix}
b_{k1} & b_{k2} \\
b_{k1} & b_{l2}
\end{matrix} \right) > \frac{1}{\sqrt n}.
\end{equation*}

\textbf{Case B.} All the norms of the rows of the matrix $A$ are greater than $\frac{1}{\sqrt{n}}$:
\begin{equation}\label{assm}
a_{i1}^2 + a_{i2}^2 > \frac{1}{n} \text{ for all } i = 1, \ldots, n.
\end{equation}

Denote the rows of $A$ by $a_1, \ldots, a_n \in \mathbb{R}^2$ and their squares by $w_1, \ldots, w_n \in \mathbb{R}^2$.

By assumption \eqref{assm},
\begin{equation*}
|w_i| = ||a_i||^2 > \frac{1}{n}.
\end{equation*}

Moreover, since $A^TA = I$, the following holds
\begin{equation}\label{ws}
\sum_{i = 1}^{n} w_i = 0, \quad \sum_{i = 1}^{n} | w_i | = 2.
\end{equation}

Also, similar to \cite{SP2026} we have
\begin{equation}\label{sqdot}
(a_i, a_j)^2 = \frac{1}{2} | w_i | | w_j | + \frac{1}{2} (w_i, w_j), \text{ for all } \, i, j = 1, \ldots, n.
\end{equation}

Now we will prove that there exist $i$ and $j$, $i \neq j$ such that
\begin{equation}\label{trgt}
(a_i, a_j)^2 \leq (| w_i | - \frac{1}{n}) (| w_j | - \frac{1}{n}).
\end{equation}

Applying \eqref{sqdot} we get the equivalent inequality
\begin{equation*}
(w_i, w_j) + \frac{2}{n^2} \leq (| w_i | - \frac{2}{n}) (| w_j | - \frac{2}{n}).
\end{equation*}

Denote the values $|w_i| - \frac{2}{n}, i = 1, \ldots, n$ by $z_1, \ldots, z_n$. For them the following holds
\begin{equation*}
\sum_{i = 1}^{n} z_i = \sum_{i = 1}^{n} |w_i| - 2 = 0.
\end{equation*}

Thus, \eqref{trgt} is equivalent to
\begin{equation}\label{trgtnew}
(w_i, w_j) - z_i z_j + \frac{2}{n^2} \leq 0.
\end{equation}

Denote the left-hand side of \eqref{trgtnew} by $M_{ij}$ and assume that $M_{ij} > 0$ for any $i \neq j$. Together with \eqref{assm} this implies that
\begin{equation}\label{pos}
M_{ij} > 0, \text{ for all } \, i, j = 1, \ldots, n.
\end{equation}

Denote the $n \times 2$ matrix composed by $w_1, \ldots, w_n$ as rows by $W$ and the column vector of coefficients $z_1, \ldots, z_n$ by $z$.

Then introduce an $n \times n$ matrix $G = W W^T - z z^T$. Since $\mathrm{rank} \, W \leq 2$,
\begin{equation}\label{ipls}
i_+(G) \leq 2,
\end{equation}
where $i_+(G)$ is the number of positive eigenvalues of $G$.

Computing the trace
\begin{equation}\label{trc}
\mathrm{tr} \, G = \sum_{i = 1}^{n} (| w_i |^2 - (|w_i| - \frac{2}{n})^2) = \frac{4}{n}
\end{equation}
and combining this with \eqref{ipls} we get
\begin{equation}\label{l1ineq}
\lambda_1(G) \geq \frac{2}{n},
\end{equation}
where $\lambda_1(G)$ is the largest eigenvalue of $G$.

Now, define the matrix $M = (M_{ij})_{i,j=1}^n = G + \frac{2}{n^2} E$, where $E$ is the $n \times n$ all-ones matrix.

For the all-ones vector $\mathbf{1} \in \mathbb{R}^{n \times 1}$ we have
\begin{equation}\label{f1}
G \mathbf{1} = (W W^T - z z^T) \mathbf{1} = \mathbf{0}.
\end{equation}
Therefore, $\mathrm{Span}(\mathbf{1})$ and $\mathrm{Span}(\mathbf{1})^\perp$ (further, simply $\mathbf{1}^{\perp}$) are invariant spaces of~$G$, and by \eqref{l1ineq}
\begin{equation}\label{f2}
\lambda_1(G|_{\mathbf{1}^{\perp}}) \geq \frac{2}{n}.
\end{equation}

On the other hand,
\begin{equation}\label{f3}
M \mathbf{1} = (G + \frac{2}{n^2} E) \mathbf{1} = \frac{2}{n^2} E \mathbf{1} = \frac{2}{n} \mathbf{1},
\end{equation}
and by \eqref{pos} $M$ is entrywise positive. Then by the Perron-Frobenius theorem we have
\begin{equation}\label{f4}
\lambda_1(M|_{\mathbf{1}^{\perp}}) < \frac{2}{n},
\end{equation}
which contradicts \eqref{f2} because of the fact that $G|_{\mathbf{1}^{\perp}} = M|_{\mathbf{1}^{\perp}}$. Therefore, the assumption that \eqref{trgtnew} does not hold for any $i \neq j$ is wrong, and we have proved that \eqref{trgt} is satisfied for some distinct $i$ and $j$.

Let us consider the corresponding submatrix
\begin{equation*}
A_{ij} = \left( \begin{matrix}
a_{i1} & a_{i2} \\
a_{j1} & a_{j2}
\end{matrix} \right).
\end{equation*}

The characteristic polynomial of its Gram matrix $\Gamma_{ij} = A_{ij} A_{ij}^T$ is 
\begin{equation*}
\chi(\lambda) = (| a_i |^2 - \lambda) (| a_j |^2 - \lambda) - (a_i, a_j)^2.
\end{equation*}

In Case B, $\mathrm{tr} \, \Gamma_{ij} = |a_i|^2 + |a_j|^2 > 2/n$. Hence $\lambda_1(\Gamma_{ij}) > 1/n$.

Since by \eqref{trgt} $\chi(\frac{1}{n}) \geq 0$ and $\chi(\frac{1}{n}) = (\lambda_1(\Gamma_{ij}) - \frac{1}{n})(\lambda_2(\Gamma_{ij}) - \frac{1}{n})$, we get
\begin{equation*}
\lambda_2(\Gamma_{ij}) \geq \frac{1}{n}
\end{equation*}
and the desired inequality
\begin{equation*}
||A_{ij}^{-1}||_2 \leq \sqrt{n}.
\end{equation*}

Let us now analyze the equality case. It consists in the fact that 
\begin{equation}\label{mx}
\max_{i \neq j} \lambda_2(\Gamma_{ij}) = \frac{1}{n}.
\end{equation}

In terms of matrix $M$, \eqref{mx} means that $M$ is entrywise nonnegative and has at least one zero entry.

Repeating the arguments from \eqref{f1})-\eqref{f3} and replacing inequality \eqref{f4} with its non-strict version, we get
\begin{equation*}
\lambda_1(G|_{\mathbf{1}^{\perp}}) = \frac{2}{n}.
\end{equation*}

Together with \eqref{ipls} and \eqref{trc}, this gives the following
\begin{equation*}
\lambda_1(G) = \lambda_2(G) = \frac{2}{n}, \; \lambda_3(G) = \ldots = \lambda_n(G) = 0.
\end{equation*}

Since $G|_{\mathbf{1}^{\perp}} = M|_{\mathbf{1}^{\perp}}$ and by \eqref{f3},
\begin{equation*}
\lambda_1(M) = \lambda_2(M) = \lambda_3(M) = \frac{2}{n}, \; \lambda_4(M) = \ldots = \lambda_n(M) = 0.
\end{equation*}

%%%%%%%%%%%%%%%%%%%%%%%%%%%%%%%%%%%%%%%%%%%%%%%%%%%

Since $M$ is real symmetric and entrywise nonnegative, after a simultaneous permutation of rows and columns it is block diagonal with irreducible nonnegative diagonal blocks. By the Perron-Frobenius theorem, each nonzero irreducible block has a simple positive dominant eigenvalue (see, for example, \cite{BP1994}). Since the spectrum of $M$ contains exactly three positive eigenvalues, all equal to $2/n$, and since $M_{ii} > 0$ for every $i$, there are exactly three such blocks, each of rank one.

Therefore, $M = u_1 u_1^T + u_2 u_2^T + u_3 u_3^T$ for certain nonnegative orthogonal vectors $u_1, u_2, u_3 \in \mathbb{R}^{n \times 1}$. Since $\mathbf{1} \in \mathrm{Span}(u_1, u_2, u_3)$, we can state that up to a simultaneous permutation of rows and columns
\begin{equation}\label{blck}
M = \left( \begin{matrix}
\frac{2}{pn}E_p & 0 & 0 \\
0 & \frac{2}{qn}E_q & 0 \\
0 & 0 & \frac{2}{rn}E_r
\end{matrix} \right),
\end{equation}
where $p + q + r = n$ and $E_p, E_q, E_r$ are the all-ones matrices of the corresponding sizes.

For vector $w_i$ corresponding to the first block (cluster) of \eqref{blck} we have
\begin{equation*}
|w_i|^2 - (|w_i| - \frac{2}{n})^2 + \frac{2}{n^2} = \frac{2}{pn},
\end{equation*}
and thus
\begin{equation*}
|w_i| = \frac{1}{2n} + \frac{1}{2p}.
\end{equation*}

If $w_j$ corresponds to the same block, the following holds
\begin{equation*}
(w_i, w_j) - (|w_i| - \frac{2}{n})(|w_j| - \frac{2}{n}) + \frac{2}{n^2} = \frac{2}{pn},
\end{equation*}
\begin{equation*}
(w_i, w_j) = |w_i||w_j|.
\end{equation*}
And since $|w_i| = |w_j|$,
\begin{equation*}
w_i = w_j,
\end{equation*}
which establishes the necessity of the formulated equality condition.

Conversely, let $w_1, \ldots, w_n$ be the squares of the rows of $A$ possessing the described three-cluster structure. Considering the triangle formed by the sums of vectors of the three given clusters and applying the cosine rule, we obtain formula \eqref{blck} and thus the sufficiency of the formulated equality condition.

Proposition \ref{pr} is proven.

\section{Acknowledgements}

I would like to thank Igor Makhlin, Stanislav Budzinskiy, Aleksei Ustimenko and the authors of \cite{GTZ1997} and \cite{SP2026} for fruitful discussions.

\bibliographystyle{plain}
\bibliography{lit}

\end{document}